\documentclass{amsart}
\usepackage{epsfig}
\usepackage{amssymb,multicol}
\theoremstyle{plain}
\newtheorem{thm}{Theorem}[section]
\newtheorem{lemma}[thm]{Lemma}
\newtheorem{cor}[thm]{Corollary}
\newtheorem{conj}[thm]{Conjecture}

\newtheorem{que}[thm]{Question}
\newtheorem{prop}[thm]{Proposition}
\theoremstyle{definition}
\newtheorem{defn}[thm]{Definition}

\newtheorem{remark}[thm]{Remark}
\begin{document}

\title[Locally $G$-homogeneous Busemann $G$-spaces]{Locally $G$-homogeneous Busemann $G$-spaces}
\author[
V. N. Berestovski\u\i,
D. M. Halverson, and
D. Repov\v s
]
{
V. N. Berestovski\u\i,
Denise M. Halverson, and
Du\v san Repov\v s
}

\address{Omsk Branch of Sobolev Institute
of Mathematics SD RAS,
Pevtsova 13, Omsk, Russia 644099}
\email{berestov@ofim.oscsbras.ru}

\address{Department of Mathematics,
Brigham Young University,
Provo, UT 84602}
\email{deniseh@math.byu.edu}

\address{Faculty of Mathematics and Physics, and
Faculty of Education,
University of Ljubljana,
P. O. Box 2964,
Ljubljana, Slovenia 1001}
\email{dusan.repovs@guest.arnes.si}



\begin{abstract}
We present short proofs of all known topological properties of general
Busemann $G$-spaces (at present no other property is known for dimensions 
more than four). We prove that
all small metric spheres in locally $G$-homogeneous Busemann $G$-spaces
are homeomorphic and strongly topologically homogeneous.
This is a key result in the context of the classical Busemann conjecture
concerning the characterization of topological manifolds, which asserts
that every $n$-dimensional Busemann $G$-space is a topological $n$-manifold.
We also prove that every Busemann $G$-space which is uniformly locally
$G$-homogeneous on an orbal subset must be finite-dimensional.
\end{abstract}

\date{\today}
\thanks{The corresponding author is {\it Du\v san Repov\v s}. His email address is {\it dusan.repovs@guest.arnes.si}}
\maketitle

\section{Introduction}

A metric space is said to be a {\it Busemann $G$-space} if it
satisfies four basic axioms that, among other things, imply that
the space is a complete geodesic space (a precise definition will be
given later). This class of spaces was introduced in 1942 by Herbert
Busemann \cite{Busemann1}--\cite{Busemann3} in an attempt to
present Finsler manifolds in simple geometric terms. Subsequent
investigations in \textit{the geometry of geodesics} were summarized
in \cite{Busemann6,Busemann-Phadke3}. Busemann and Phadke
\cite{Busemann-Phadke2} introduced and studied an interesting
generalization of Busemann $G$-spaces. Their survey
\cite{Busemann-Phadke3} can be considered as a  testament
to researchers in the area
of the geometry of geodesics.

In the present paper we give
short proofs of topological properties of general finite-dimensional
Busemann $G$-spaces -- no other property is known at present without
specification of the dimension. A new result among them is that
small spheres in every $n\geq 3$-dimensional Busemann $G$-spaces are
simply connected.

At present, the answer to the Busemann question \cite{Busemann3}, which asks if
every Busemann $G$-space
must be finite-di\-men\-sio\-nal is still unknown.
Heretofore, the best known result 
was that  this is true for every Busemann $G$-space with
small geodesically convex balls near some point
\cite{Berestovskii1}.
The latter condition is satisfied at every point of a Busemann $G$-space $X$ if $X$ has
nonpositive curvature in the Busemann sense \cite{Busemann3,
Busemann4}, which means that in small triangles the length of the
midsegment is no more than the half of the length of the
corresponding side.

Note that nowadays many
authors
apply the term \textit{Busemann space} to a
geodesic space with a  local or global condition of nonpositive
curvature in the Busemann sense \cite{Papadopoulos}. However, there
exist metrically homogeneous Finsler 2-manifolds among the so-called
\textit{quasihyperbolic planes}, which are Busemann $G$-spaces with
no geodesically convex balls of positive radius (the assertion was stated in \cite{Busemann5} and
proved in \cite{Gribanova}). Weaker assertions have been proved in \cite{Busemann-Phadke1}.
 In this paper we shall
 generalize the result from \cite{Berestovskii1} stated above
to all Busemann $G$-spaces which are
uniformly locally $G$-homogeneous on an
orbal subset. It is unknown whether every Busemann $G$-space
satisfies this property.

Busemann conjectured that for all $n < \infty$, every
$n$-dimensional  Busemann $G$-space is a topological $n$-manifold.
The $(n \le 4)$-dimensional Busemann $G$-spaces are known to be
topological $n$-manifolds (cf. \cite{Busemann3,Krakus,Thurston}).
The {\it Busemann conjecture} is also known to be true in all
dimensions 
under the additional hypothesis that the Aleksandrov curvature is
bounded either {\it  from below}  or {\it from above}; such spaces
are even Riemannian (hence Finsler) manifolds with continuous metric
tensors \cite{Berestovskii2,Berestovskii3}. There are other
additional conditions which guarantee that a Busemann $G$-space is a
topological manifold, or even a Finsler space with a continuous
metric function \cite{Busemann6,Pogorelov}. We shall discuss
these results more in details later in the paper. However, this
classical problem, now over half a century old, has still not been
solved in its complete generality. For more on the Busemann  conjecture see
the recent
survey \cite{HaRSurvey}.

A finite-dimensional normed
vector space $(V,||\cdot||)$ is
a Busemann $G$-space if and only if its
closed balls of positive radius are strongly convex in
the affine sense, i.e. they are convex and
their boundary spheres do not contain non-trivial affine segments.
Under this condition, its
shortest arcs are exactly affine segments.
On the other hand,
$V$ has the Aleksandrov curvature bounded from above or below if and
only if $V$ is isometric to the Euclidean space \cite{ABN}. Therefore there exist
Busemann $G$-spaces with geodesically (strongly) convex balls which
do not have the Aleksandrov curvature bounded from above or below.
Note also that every normed vector space $(V,||\cdot||)$ is a
\textit{space with distinguished geodesics} in the sense of
\cite{Busemann-Phadke2}.

Let us observe that we shall, as did Busemann, assume that a
\textit{Finsler manifold} is a finite-dimensional
$C^1$-differentiable manifold $M$ with a continuous norm $F$ on its
tangent bundle $TM$. However, it should be noted that usually the
 Finsler geometry experts, including Finsler \cite{Finsler}
himself,  generally require additional conditions for the function
$F.$ There are no known examples of Busemann $G$-spaces which are
topological manifolds but fail to be Finsler manifolds. However, every
metrically homogeneous Busemann $G$-space is a homogeneous space of
a (connected) Lie group by its compact subgroup, and hence a
topological manifold \cite{Berestovskii4,Szenthe}. It seems that
every such space should be a Finsler manifold.
This
has actually
been proved for dimensions 2 and 3 (cf. \cite{Berestovskii5,Berestovskii6}),
whereas every  one-dimensional Busemann $G$-space is always
a Riemannian (hence Finsler)
manifold.

Pogorelov \cite{Pogorelov} proved  that a Finsler manifold $M$ with
a "strictly convex" metric function $F$ of the class $C^{1,1}$ is a
Busemann $G$-space, and moreover, that this degree of regularity
cannot be weakened. Namely, for any $\alpha<1$ there exist Finsler
manifolds with a strictly convex metric function of the class
$C^{1,\alpha}$ which are not Busemann $G$-spaces. This result
substantially improves upon
an ealier result of Busemann and Mayer
\cite{Busemann-Mayer} -- they proved the first statement above for
$C^3$-functions $F.$ 
Note that  three versions of
"strict convexity" were used
in \cite{Pogorelov}. However, the discussion in the
previous paragraph implies that there are Busemann-Finsler
$G$-spaces $V$ with metric function $F=||\cdot||$ which are not
differentiable and not strictly convex for two of the three versions
of this notion.

Pogorelov also proved in some sense the converse assertion: if in a
Busemann $G$-space the intersecting shortest curves have a certain
slope to each other which continuously depends on these shortest
curves, then such a $G$-space is a Finsler space with a continuous
metric function. Similar results were proved by Busemann
\cite{Busemann6}: if a $G$-space is "continuously differentiable and
regular" at one point then it is a topological manifold (cf. (9) on
p. 24 in \cite{Busemann6}). Busemann stated that regularity
condition can be avoided. In fact, it is more or less clear that if
a G-space is continuously differentiable at every one of its points
then it  is isometric to a Finsler space with a continuous metric
function.

The Busemann conjecture is a special case of another classical
conjecture, the {\it Bing-Borsuk conjecture} \cite{Bing-Borsuk}. A
topological space $X$ is said to be {\it topologically homogeneous}
if for any two  points $x_{1}, x_{2} \in X$, there is a
homeomorphism of $X$ onto itself taking $x_{1}$ to $x_{2}$. It is a
classical result that all connected manifolds without boundary are
topologically homogeneous. The Bing-Borsuk conjecture states that
all finite-dimensional topologically homogeneous ANR-spaces are manifolds.

It is well-known that Busemann $G$-spaces are topologically
homogeneous \cite{Thurston} (see also in the present paper) 
and locally
contractible, so they are ANR-spaces if they are finite-dimensional
\cite{Kuratowski}. Thus, even though it is hardly believable that
the Busemann conjecture is not true, a counterexample to it would
settle the Bing-Borsuk conjecture in  the negative. On the other
hand, a proof of the Busemann conjecture may shed some light on the
Bing-Borsuk conjecture.

Implied from the basic geometric properties is that every small
metric ball in a Busemann $G$-space is the cone from its center over
its boundary.  As a result of topological homogeneity and this cone
structure, a Busemann $G$-space $M$ is a manifold if and only if all
small metric spheres in $M$ are codimension one manifold factors.
Thus the characterization of small metric spheres is of vital
importance in addressing the question of whether high-dimensional
Busemann $G$-spaces are manifolds in general. Several geometric
properties which imply that a given topological space is a
codimension one manifold factor can be found in \cite{Halverson1,
Halverson2, Halverson3, HaR}.

Demonstrating the topological homogeneity of  small metric spheres
is a key step to proving the general case of the Busemann conjecture
\cite{HaRSurvey}. In this paper we introduce a special type of
homogeneity property, the so-called local $G$-homogeneity. Local
$G$-homogeneity essentially requires that any  sufficiently small
metric ball  can be represented as a cone over any point which is
sufficiently close to its center, the cone lines being geodesics.
We shall demonstrate that  in Busemann $G$-spaces the property of
local $G$-homogeneity implies that all  sufficiently small metric
spheres are mutually homeomorphic and topologically homogeneous.

The following are main results of the present paper:

\begin{thm}\label{main}
Suppose $X$ is a locally $G$-homogeneous Busemann $G$-space.  Then
sufficiently small metric spheres  in $X$ are (strongly)
topologically homogeneous.
\end{thm}

\begin{thm}\label{fdim}
Suppose $X$ is a Busemann $G$-space, uniformly locally
$G$-homogeneous on an orbal subset.  Then $X$ is finite-dimensional.
\end{thm}

\begin{thm}\label{ex}
There exists a Busemann $G$-space $X$
with the following properties:
\begin{enumerate}
\item $X$ is uniformly locally $G$-homogeneous
on an orbal subset;
\item $X$ is locally $G$-homogeneous; and
\item $X$ has no convex metric balls of positive radius.
\end{enumerate}
\end{thm}

In the Epilogue we shall
collect some unsolved questions.

\section{Preliminaries}

\begin{defn}
\label{bus}
Let $(X,d)$ be a metric space.
$X$ is said to be a {\it Busemann $G$-space}
provided it satisfies the following axioms of Busemann:

\begin{enumerate}
\item[(i)] \emph{Menger Convexity:} Given
distinct points $x,y \in X$, there is a point $z \in X-\{x,y\}$
such that $d(x,z) + d(z,y) = d(x,y)$;

\item[(ii)] \emph{Finite Compactness:}
Every $d$-bounded infinite set has an accumulation point;

\item[(iii)] \emph{Local Extendibility:}
For
every point
$w \in X$, there exists a radius $\rho_w>0$, such
that for any pair of distinct points $x,y$ in
the open ball $U(w,\rho_w)$, there is a point
$z \in U(w,\rho_w)- \{x,y\}$ such that $d(x,y) + d(y,z) = d(x,z)$; and

\item[(iv)] \emph{Uniqueness of Extension:}
Given distinct  points
$x,y \in X$, if there are points $z_1, z_2 \in X$ for which both
equalities
$$d(x,y) + d(y,z_i) = d(x,z_i) \quad \text{for }
i=1,2, $$ and $$d(y,z_1) = d(y,z_2)$$ hold, then $z_1 = z_2$.
\end{enumerate}
\end{defn}

\begin{remark}  \label{property}
From these basic properties, a rich structure on a Busemann
$G$-space can be derived.
If $(X,d)$ is a Busemann $G$-space and $w \in X$ is any point, then
$(X,d)$ satisfies the following properties:
\begin{itemize}
  \item {\it Complete Inner Metric:} $(X,d)$ is a locally compact
  complete inner metric space;
  \item {\it Existence of Geodesics:} Any two points in $X$ can be
  joined by a geodesic;
  \item {\it Local Uniqueness of Joins:}  Any two points $x,y$
  in $U(w,\rho_w)$ can be joined by a unique shortest geodesic in $X$;
  \item {\it Local Cones:} The closed ball $B(w,r), 0<r<\rho_w,$
  is homeomorphic to the cone over its boundary (cf. Proposition \ref{cone}
  below);
  \item {\it Topological Homogeneity:} Every Busemann $G$-space is topologically
  homogeneous. Moreover, topological homogeneity homeomorphism can be chosen to be
  isotopic to the identity (cf. Theorem \ref{hom} and Corollary \ref{homo} below).
  \end{itemize}

\end{remark}

Busemann \cite{Busemann3} has proposed the following conjecture
which still remains open in dimensions $n \ge 5$:

\begin{conj}[Busemann Conjecture]\label{BC}
Every $n$-dimensional Busemann $G$-space,
$n \in \mathbb{N}$, is a topological $n$-manifold.
\end{conj}

In this paper we shall show (cf. Theorem~\ref{main}) that stably
visible  metric spheres in any Busemann $G$-space are strongly
topologically homogeneous (cf. Definitions \ref{local G-homog} and \ref{strong homog}).

\section{
Topological properties of
finite-dimensional Busemann $G$-spaces}

Thurston \cite{Thurston} has shown that small metric spheres  in any
$n$-dimensional Busemann $G$-space are homology $(n-1)$-manifolds
(throughout this paper we are working only with {\it singular}
homology with $\mathbb{Z}$ coefficients).

In this section, using only old results, known from topological
literature until 1963, we shall briefly prove all known topological
properties, in particular the assertion above due to  Thurston, for
arbitrary finite-dimensional Busemann $G$-spaces.

For convenience we shall use the following notations and definition.
Let $I$ denote the unit interval $[0,1]$.  $B(x,r)$ shall denote the
closed ball of radius $r$ centered on $x$ and $U(x,r)$ shall denote
the open ball of radius $r$ centered on $x$.

\begin{defn} If $x$, $y$, and $z$ are distinct points in a Busemann
$G$-space and
$$d(x,y) + d(y,z) = d(x,z)$$ we say that \emph{$y$ lies between $x$
and $z$} and denote this by $x - y - z$.
\end{defn}

Let $(X,d)$ be any Busemann $G$-space. For a point $w\in X$ we
denote by $\rho(w)$ the supremum of all numbers $\rho_w$ which
satisfy the condition (iii) from Definition \ref{bus}.

The following statement is an easy consequence of definitions.

\begin{lemma}
\label{rho}
The function $\rho(w)=+\infty$ for all points $w\in X$ or
\begin{equation}
\label{ineqrho}
|\rho(x)-\rho(y)|\leq d(x,y) \quad\mbox{for all}\quad x,y \in X.
\end{equation}
\end{lemma}

One can easily  sequentially
prove the
assertions of
the next proposition.

\begin{prop}
\label{cone} Suppose that $0<r<\rho(x).$ Let $S:=S(x,r)$  and
$B:=B(x,r)$.  Then
\begin{itemize}
  \item For every point $y$ in the sphere $S$ there is a
  unique
shortest arc (segment) $\overline{xy}$, joining points $x$ and $y;$

\item Segment $\overline{xy}$ continuously depends on point $y\in S$
in the sense that the real-valued function  $\phi:S\times S \to \mathbb{R}$ where
$$\phi(y_1,y_2):=d_H(\overline{xy_1}, \overline{xy_2})$$
where $d_H$ denotes the Hausdorff distance (between compact subsets),
is continuous;

\item Every point $z\in B -\{x\}$ lies on a
unique segment $\overline{xy}, y=y(z)\in S;$ and

\item Let $c:S\times I \rightarrow C(S)$ be the canonical map of $S$ onto its
cone, identifying all points $(y,0)\in S\times I$ to the vertex $v$ of the cone.
Then
the map $f:B\to C(S),$ defined by the formula
$$ f(z) = \left\{
            \begin{array}{ll}
              c\left(y(z), \frac{d(x,z)}{ d(x,y(z))} \right), &   z\in B-\{x\}\\
              v, &   z=x
            \end{array}
          \right. $$
is a homeomorphism.
\end{itemize}

\end{prop}

\begin{remark}
\label{gc}
Note that the first two assertions of
Proposition \ref{cone} remain
true if we change $x$ by any  other point $x'\in B-S.$ If for some
point $x'\in B-S,$ every segment $\overline{x'y},$ where $y\in S,$
intersects $S$ only at the point $y$ (in other words, \textit{the
sphere $S$ is visible from the point $x'$}), then the last statement
of Proposition \ref{cone} is true after replacement $x$ by $x'.$ We
shall say in this case that the above map $f$ defines \textit{the
canonical structure of geodesic cone on $B$ with the vertex $x$ (or
$x'$)} and the closed ball $B$ is (geodesically) \textit{star-like
with respect to the point $x$ (or $x'$)}.
\end{remark}

\begin{lemma}
\label{int} For any two numbers $r_1,r_2\in (0,1)\subset I,$ there
is an isotopy $h:I\times I\rightarrow I$ fixed on $\{0,1\}$ such
that $h(\cdot,0)=\text{id}_{I}$ and $h(r_2,1)=r_1.$
\end{lemma}

\begin{proof}
We can suppose that $0<r_1<r_2<1.$ Then there is a unique real
number $\alpha>1$ such that $r_1=r_2^{\alpha}.$ The map
$h(r,t)=r^{1+t(\alpha-1)}$ is  the required isotopy.
\end{proof}

\begin{prop}
\label{cis} Let $C:=C(S)$ be a cone on a topological space $S,$
$r_1,r_2\in (0,1),$ and $x=c(s_0,r_2)$, $y=c(s_0,r_1)$ for some
$s_0\in S.$ Then there is an isotopy $H:C\times I\rightarrow C$
fixing the base and the vertex of the cone $C$ such that
$H(\cdot,0)=id_{C}$ and $H(x,1)=y.$
\end{prop}

\begin{proof}
The required isotopy is defined by the  formula
$H(c(s,r),t)=c(s,h(r,t)),$ where $s\in S,r\in I$ and $h$ is the
isotopy from  Lemma \ref{int}.
\end{proof}

\begin{prop}
\label{bis} Let $B:=B(x_0,r)$, $ x_0\in X$, $0<r<\rho(x_0)$,
$S:=S(x_0,r).$ Then for any two points $w,z$ lying inside some
segment $\overline{x_0 s_0},$ where $s_0\in S,$ there is an isotopy
$H':B\times I\rightarrow B$ fixing $S$ and $x_0$ such that
$H'(\cdot,0)=id_B$ and $H'(w,1)=z.$
\end{prop}

\begin{proof}
By Proposition \ref{cone}, there is a homeomorphism
$f:C=C(S)\rightarrow B.$  Then there are numbers $r_1,r_2\in (0,1)$
such that $f(s_0,r_2)=w, f(s_0,r_1)=z.$ We define the isotopy $H'$
by the formula $H'=f\circ H\circ f^{-1},$ where $H$ is the isotopy
from Proposition \ref{cis}.
\end{proof}

\begin{lemma} \label{lemma}
\label{isotins} Let $B=B(x_0,r)$, $x_0\in X,0<r<\rho(x_0)$,
$S=S(x_0,r)$. Then for every point $s_0\in S$ there are a point $y$
and an isotopy $H'':B\times I\rightarrow B$ fixing $S$
such that $x_0- y- s_0,$ $H''(\cdot,0)=id_B,$ and
$H''(x_0,1)=y.$
\end{lemma}

\begin{figure} \label{thm}
    \begin{center}
\epsfig{file=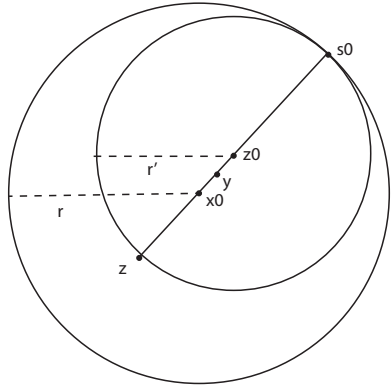, width=2.0in, height=2.0in}
        \caption{Sketch for Lemma \ref{lemma}}
    \end{center}
\end{figure}

\begin{proof}
There is a point $z\in U:=U(x_0,r)$ such that $s_0-x_0-z.$
Let $z_0$ be the midpoint of $\overline{s_0z}.$ Then
$z-x_0-z_0$ because $d(s_0,z)<2r.$ Thus there is a point $y$ such
that $z_0-y-x_0.$  Let $ S':=S(z_0,r')$ and $B':=B(z_0,r'),$ where
$r'=d(z_0,s_0)$. Note
that $ z\in S'$. By the Triangle inequality, $B'\subset B.$ Thus
$B'$ is a geodesic cone over $S'$ with the vertex $z_0.$
By Proposition \ref{bis}, there is an isotopy
$H:B'\times I\rightarrow B'$ fixing $S'$ and $z_0$ such that
$H(\cdot,0)=id_{B'}$ and $H(x_0,1)=y.$  So we can extend the
isotopy $H$ to the required one
$H''$ on $B$, fixing $H$ outside of $B'.$
\end{proof}

\begin{thm}
\label{bbis} Let $B=B(x_0,r)$, $x_0\in X$, $0<r<\rho(x_0).$ Then for
every point $x \in U$ there is an isotopy $H:B\times I\rightarrow B$
fixing $S:=S(x_0,r)$ such that $H(\cdot,0)=id_B$ and $H(x_0,1)=x.$
\end{thm}

\begin{proof}
We need only to
consider the case when $x\neq x_0.$ Then there is
unique point $s_0\in S$ such $x_0-x-s_0.$ By Lemma \ref{isotins}
there  exist
a point $y$ and an isotopy $H'':B\times I\rightarrow B$ fixing $S$
such that $x_0- y- s_0,$ $H''(\cdot,0)=id_B,$ and
$H''(x_0,1)=y.$ By Proposition \ref{bis},
there is an isotopy $H':B\times I\rightarrow B$ fixing $S$ such that
$H'(\cdot,0)=id_B$ and $H'(y,1)=x.$ Now the ``composition'' $H$ of
isotopies $H''$ and $H'$ gives us the required isotopy.
\end{proof}

\begin{cor}
\label{homog} For an arbitrary closed ball $B:=B(x_0,r)$ in a
Busemann $G$-space, of  radius $r$, $0<r<\rho(x_0),$ and any $x
\in U(x_0,r)$ there is a homeomorphism $h: B \to B$, fixing the
sphere $S(x_0,r)$, such that $h(x_0)=x.$
\end{cor}

\begin{thm}
\label{hom} Let $x,y$ be any two points in a Busemann $G$-space,
$X.$ Then there is an isotopy $H:X\times I\rightarrow X$ such that
$H(\cdot,0)=id_{X}$ and $H(x,1)=y.$
\end{thm}

\begin{proof}
We can suppose that $x\neq y.$ Then there is a (shortest) segment
$\overline{xy}.$ By Lemma \ref{rho}, the function $\rho(w)$ is
infinite or continuous. In both cases there is a number $r>0$ such
that $r<\rho(w)$ for every point $w\in \overline{xy}.$ Then there is
a finite set $\{x_0=x,x_1, \dots, x_k=y\}$ of points in
$\overline{xy}$ such that $d(x_i,x_{i+1})<r$ for every $i=0,1,\dots,
k-1.$ Let $B_i:=B(x_i,r)$ and $S_i:=S(x_i,r)$. By Theorem
\ref{bbis}, there is an isotopy $H_{i}:B_{i}\times I\rightarrow
B_{i}$ fixing $S_{i}$ such that $H_{i}(\cdot,0)=id_{B_{i}}$ and
$H_{i}(x_i,1)=x_{i+1}$. We extend the isotopies to $X$ requiring
that $H_{i}$ fixes all points outside $B_{i}.$ Now the
``composition'' $H$ of isotopies $H_0,\dots, H_{k-1}$ gives us the
required isotopy.
\end{proof}

\begin{cor}
\label{homo}
Every Busemann $G$-space is topologically homogeneous.
\end{cor}

\begin{prop}
\label{anr}
Any finite-dimensional Busemann $G$-space $X$ is an
absolute neighborhood
retract (ANR).
\end{prop}

\begin{proof}
Clearly $X$ is arcwise connected. Proposition \ref{cone} implies
that $X$ is locally contractible. Now the statement  follows from
\cite{Kuratowski}.
\end{proof}

\begin{defn}
\label{kos}
\cite{Kosinski,Lee}
$X$ is a Kosi\'nski $r$-space provided that:

(1) $X$ is locally compact, metric, separable and finite-dimensional; and

(2) Each point of $X$ has arbitrarily small closed  neighborhoods
$U$ such that the boundary $Bd(U)$ is a strong deformation retract
of $U-y$ for each interior point $y$ of $U.$
\end{defn}

Note that Kosi\'nski \cite{Kosinski}
assumed that $X$ is compact, but
this condition can be replaced here by
the
local compactness.
As an immediate consequence of Proposition
\ref{cone} and Corollary \ref{homog} we get
the following corollary
(cf. also Theorem 3 on p. 16 in \cite{Busemann6}).

\begin{cor}
\label{kosi}
Every finite-dimensional Busemann $G$-space is a Kosi\'nski $r$-space.
\end{cor}

\begin{thm}
\label{homman}
Every $n$-dimensional Busemann $G$-space $X$
is a $\mathbb{Z}$-homology
$n$-manifold, i.e., for every point $x\in X$,
$H_k(X,X-\{x\};\mathbb{Z})\cong \mathbb{Z}$ if
$k=n$ and $H_k(X,X-\{x\};\mathbb{Z})=0$ if $k\neq n.$
\end{thm}

\begin{proof}
Alexandroff \cite{Alexandroff} proved that for any finite-dimensional (separable
metric) space its cohomological dimension over the ring $\mathbb{Z}$
coincides with its  topological (i.e. covering) dimension. The space $X$ is arcwise connected, locally
contractible and  by  Corollary
\ref{kosi}
it is a Kosi\'nski $r$-space. 
We can now simply apply results of Lee \cite{Lee}.
\end{proof}

\begin{remark}
Thurston \cite{Thurston}
proved  that any finite-dimensional
Busemann $G$-space is an ANR $\mathbb{Z}$-homology manifold, using more recent 
results of Dydak and Walsh \cite{Dydak-Walsh}.
\end{remark}

\begin{prop}
\label{spheroid}
Let $S:=S(x_0,r)$, where $0<r<  \rho(x_0),$
be any sphere in a Busemann $G$-space $(X,d),$ let $x\in S$ be any
point and denote by $x'\in S$  its (unique) "antipodal" point, i.e.
$d(x,x')=2r$. Then $\{x'\}$ is a strong deformation retract of
$S':=S-\{x\}.$
\end{prop}

\begin{proof}
Let us first consider
$S=S(x_0,r)$,
where
$0<r<\frac 12 \rho(x_0)$.
We may assume
that $X$ is not 1-dimensional. Otherwise
$S=\{x,x'\},$ and there is nothing to prove.
By hypothesis, 
$2r<\rho(x_0).$ Then for every point $z$ in $B(x_0,2r)-\{x_0\}$ the
(unique) segment $\overline{x_0z}$ or its extension to a segment
necessarily intersects the set $S'$ in a unique point, which we
shall denote by $f(z).$
By the condition on
$r,$ any two points $y,z\in S$ are joined in $X$ by a unique segment
$\overline{yz}$ (which does not necessarily lie in the closed ball
$B(x_0,r)$) and these segments continuously depend on their ends. For
every point $y\in S',$ the segment $\overline{yx'}$ does not go
through $x_0.$

As a corollary of
the Triangle inequality, every such
segment $\overline{yx'}$ lies in the ball $B(x_0,2r).$
For a point $y\in S'$ define $h_y(t), t\in
I,$ as the point on the segment $\overline{yx'}$ such that
$d(h_y(t),x')=(1-t)d(y,x').$ Now because of all what was said
before, the formula $H(y,t)=f(h_y(t))$ defines a homotopy
$H:S'\times I \rightarrow S'$ such that $H(\cdot,0)=id_{S'},
H(S',1)=\{x'\}$ and $H(x',t)=x'$ for all $t\in I.$
The proof for $0<r<  \rho(x_0)$
can be completed
by Proposition \ref{cone}.
\end{proof}

See also pp. 17-18 in \cite{Busemann6}.

\begin{thm}
\label{homsphere} Let $S:=S(x_0,r)$, where $0<r< \rho(x_0),$
be any sphere in an $n$-dimensional Busemann $G$-space
$X$. Then $S$ is a 
$\mathbb{Z}$-homology $(n-1)$-manifold and 
has the homology of the $(n-1)$-sphere.
\end{thm}

\begin{proof}
We shall
use the Eilenberg-Steenrod homology axioms \cite{Eilenberg-Steenrod}.
The case $n=1$ is trivial. So suppose that $n>1.$
Let $B:=B(x_0,r)$ and $B':=B-\{x_0\}.$ Evidently the closure of
$X-B$ is contained in $X-\{x_0\}.$ Then by the Excision axiom,
$H_k(X,X-\{x_0\})$ is isomorphic to $H_k(B,B')$ and so by Theorem
\ref{homman}, the latter group is $\mathbb{Z}$ if $k=n$ and
$0$ if $k\neq n.$ Moreover, as a corollary of Proposition
\ref{cone}, $S$ is a strong deformation retract of $B',$ while $B$
is contractible. Hence by the Homotopy axiom, $H_k(B')\cong H_k(S)$
for all $k;$ $H_k(B)=0$ for $k\neq 0$ and $H_0(B)\cong \mathbb{Z}.$

Consider the following part of the exact homology sequence for the
pair $(B,B')$:
$$\dots\rightarrow H_{k+1}(B)\rightarrow H_{k+1}(B,B')
\rightarrow H_{k}(B')\cong H_{k}(S)\rightarrow H_{k}(B)\rightarrow\dots$$
If $k>0$ then the
first and last terms are $0$ and so
$$H_k(S)\cong H_{k+1}(B,B'),$$
which is nonzero only if $k+1=n$ or $k= n-1$ and $H_{n-1}(S)\cong
H_{n}(B,B')\cong \mathbb{Z}.$ Also $H_{0}(S)\cong \mathbb{Z}$
because $S$ is arcwise connected.  This means that $S$
has the homology of the $(n-1)$-sphere.

For any point $x\in S,$ consider the following part of the 
homology exact sequence for the pair $(S,S-\{x\})=(S,S')$:
$$\dots\rightarrow H_{k+1}(S')\rightarrow H_{k+1}(S)
\rightarrow H_{k+1}(S,S')\rightarrow H_{k}(S')\rightarrow H_k(S)
\rightarrow\dots$$
If $k>0$ then $H_{k+1}(S')\cong H_k(S')=0$ by Proposition
\ref{spheroid}.  Then $H_{k+1}(S,S')\cong H_{k+1}(S)$ which is
nonzero only if $k+1=n-1$ and $H_{n-1}(S,S')\cong H_{n-1}(S)\cong
\mathbb{Z}$ by the statements above. For $k>0$ the
latter two equalities make sense only if $n>2.$ If $k=0$ then the last
arrow is an isomorphism of groups, which are both isomorphic to
$\mathbb{Z}$ because $S$  and $S'$ are arcwise connected by the
argument from the proof of Proposition \ref{spheroid}. Then
$H_1(S,S')\cong H_1(S),$ which is zero if $n>2$ and isomorphic 
to $\mathbb{Z}$ if $n=2.$ The arcwise connectivity of $S$ and $S'$
implies that $H_0(S,S')=0$ by the definition of $H_0(S,S').$ We have thus proved that $S$ is a 
homology $(n-1)$-manifold.
\end{proof}

We get the following immediate corollary (an entirely  different
proof can be found in \cite{Busemann6}):

\begin{cor}
In every finite-dimensional $G$-space $X$, every sphere
$S(x,r)$ of radius $0<r<\rho(x)$ is noncontractible.
\end{cor}

As a corollary of Theorem \ref{homman}, finite-dimensional Busemann
$G$-spaces also possess the {\it invariance of domain property},
which was first established for manifolds by Brouwer
\cite{Brouwer1,Brouwer2} and then generalized to homology manifolds
by Wilder \cite{Wilder} (cf. V\" ais\" ala \cite{Vaisala} for a
short proof):

\begin{thm}
[Invariance of Domain Theorem] Let $X$ be a finite-dimensional
Busemann $G$-space, and $h: C\rightarrow  D$  a homeomorphism of
 subsets in $X.$ Then $h$ maps ${\text int}(C)$ onto ${\text
int}(D).$ Analogous assertion is true for every sphere $S(x,r)$ if
$0<r<\rho(x).$
\end{thm}

We note that a very different argument for this theorem was given
in \cite{Busemann6}.

\begin{thm}
Let $X$ be a $n$-dimensional Busemann $G$-space where $n\geq 3.$
Then  every sphere $S=S(x,r), 0<r<\rho(x),$ is simply connected.
\end{thm}

\begin{proof}
Since any two spheres $S=S(x,r), S=S(x,r'),0<r, r'<\rho(x),$ are
homeomorphic by Proposition \ref{cone}, we may suppose in the proof
that $0<r<\rho(x)/2.$
It follows from Proposition \ref{spheroid} that it suffices 
to prove that every loop in $S$ is homotopic to a loop whose image
is  a proper subset of $S.$ Let $r_0>0$ be the minimal value of
continuous function $\rho$ (cf. Lemma \ref{rho}) on compact ball
$B(x_0,2r)$ and $l:I\rightarrow S$ any loop in $S.$ For the number
$r_1=\frac{1}{2}\min(r,r_0)$ there is $\delta>0$ such that
$d(l(s),l(s'))<r_1$ if $s,s'\in I, |s-s'|<\delta.$ Take any numbers
$s_0,s_1,\dots, s_m$ such that $s_0=0<s_1<\dots <s_m=1,$
$s_{j+1}-s_j<\delta$ for all $j=0,1,\dots m-1,$ and corresponding
points $y_i=l(s_i), i=0,1,\dots m.$ Then $y_0=y_m.$ By Triangle
inequality and choice of $r_1$ and $\delta,$ for every $j=0,1,\dots
m-1,$ there is unique segment $\overline{y_jy_{j+1}},$ and this
segment lies in $B(x,r+r_1/2).$ By the same reason, for all points
$z\in \overline{y_jy_{j+1}}$ and $l(s),$ where $s\in
\overline{s_j,s_{j+1}},$ there is unique segment $\overline{l(s)z},$
and this segment lies in $B(x,2r).$ There is a loop
$l_0:I\rightarrow B(x,r+r_1/2)$ such that $l_0(s_i)=l(s_i)$ for all
$i=0,1,\dots m$ and the restriction of $l_0$ to every segment
$\overline{s_j,s_{j+1}}; j=0,1,\dots,m-1,$ is a parametrization of
the segment $\overline{y_jy_{j+1}}.$ By arguments above, for every
number $s\in I$ we can define unique path $h_s(t), t\in I,$ in
$B(x,2r)$ such that $h_s(t)$ is the point on unique segment
$\overline{l(s)l_0(s)}$ with condition
$d(h_s(t),l(s))=td(l(s),l_0(s)).$  Now we can define desired
homotopy $H:I\times I\rightarrow S$ by formula $H(s,t)=f(h_s(t)),$
where $f$ is defined in the proof of Proposition \ref{spheroid}.
Indeed, it is clear that the mapping $f$ homeomorphically sends every
segment $\overline{yz}$ with ends $y,z\in S,$ such that $z$ is not
antipodal to $y$ with respect to $x,$ onto its image in $S.$ For this
reason the image of loop $l_1(s)=H(s,1)$ is no more than
1-dimensional because the image of the restriction of $l_1$ to every
segment $\overline{s_j,s_{j+1}},j=0,1,\dots,m-1,$ is equal to
$f(\overline{y_jy_{j+1}})$ which is homeomorphic to the segment
$\overline{y_jy_{j+1}}.$ Since $S$ has topological dimension $\leq
n$ and is a homology $(n-1)$-manifold by Theorem \ref{homsphere},
its topological dimension is $(n-1)\geq 2.$ Thus $l_1(I)\not = S.$
\end{proof}

\begin{remark}
This theorem would also hold for {\it infinite-dimensional}
Busemann $G$-spaces should they exist.
\end{remark}

\section{Local $G$-homogeneity}

In this section we shall introduce some basic terminology and facts:

\begin{defn}
A set $Z$ in a metric space $X$ is said to be \emph{starlike with
respect to $x \in \text{int} (Z)$} if $x$ is  joinable 
with every
point in the boundary $\partial Z$ by unique shortest geodesic
(segment) and $Z$ is the geodesic cone over   $\partial Z$ with cone
point $x$. In particular, $Z = \bigcup \{ \overline{xz} \ | \ z \in
\partial Z\}$  and if $z,z' \in \partial Z$ are different, then
$\overline{xz} \cap \overline{xz'} = \{x\}$.
\end{defn}

It follows from Proposition \ref{cone} that in Busemann $G$-spaces,
all metric balls $B(x,r), 0<r<\rho(x),$ are starlike with respect to their
centers.

\begin{defn} \label{stably starlike}
A set $Z$ in a metric space is said to be \emph{stably starlike} at
a point  $x$ if there is a $\delta > 0$ such that $Z$ is star-like
with respect to any point $y \in B(x,\delta)$.
\end{defn}

\begin{defn}\label{local G-homog}
A metric space $X$ is said to be \emph{locally $G$-homogeneous} if
for every point $x \in X$, there is a radius $\varepsilon > 0$ such
that $\varepsilon< \rho(x)$ and the ball $B(x,\varepsilon)$ is stably starlike at $x$ (or, in other
words, the sphere $S(x,\varepsilon)$ is stably visible at $x$).
\end{defn}

\begin{remark}
\label{convexity} The condition that the ball $B:=B(x,\varepsilon)$
in a Busemann $G$-space $X$ is metrically strongly  convex, i.e. any
two points $y,z\in B$ are joinable
by a
unique segment $\overline{yz}$ in $X$ and
this segment, except maybe for points $y$ and $z$, is contained in
$\text{ int}(B)$, implies the assertion that $B$ is
starlike with respect to every point in $\text{ int}(B)$. As a
corollary, a Busemann $G$-space $X,$ having for each point $x \in X$
a metrically strongly convex closed ball of positive radius with the
center $x,$ is locally $G$-homogeneous.
\end{remark}

The terminology locally $G$-homogeneous was chosen to signify that an
autohomeomorphism of $X$ fixed outside of $B(x,\varepsilon)$ taking
 $x$ to  a nearby point $y$  can be chosen to preserve cone lines in
the sense that if $z \in S(x,\varepsilon)$, then $\overline{xz} \to
\overline{yz}$. Although all Busemann $G$-spaces are topologically
homogeneous, it is unknown whether  all Busemann $G$-spaces
are also locally $G$-homogeneous.

\begin{defn}\label{strong homog}
A space $X$ is said to be \emph{strongly topologically homogeneous}
if for any two points $x,y \in X$ and path $\alpha: [0,1] \to X$
such that $\alpha(0) = x$ and $\alpha(1) = y$, the map $h: \{x\}
\times [0,1] \to X; h(x,t) = \alpha(t)$ is an ambient isotopy, i.e.
there is an isotopy $H:X \times [0,1] \to X$ such that $H|_{\{x\}
\times [0,1]} = h$.
\end{defn}

\section{Topological homogeneity}

In this section we demonstrate the first version of our main result,
the topological homogeneity of sufficiently small metric spheres in
locally $G$-homogeneous Busemann $G$-spaces.

We now define two types of maps that are key to our proof and
establish their continuity.  We begin by citing the following well
known result:

\begin{prop} \label{prop-antipode}
Suppose that $X$ is a Busemann $G$-space and $B(x,r) \subset X,$ where
$0<r<\rho(x)$. Then for
each point $z \in S(x,r'),$ where $0<r'\leq r,$ there exists a unique
point $z' \in S(x,r')$ such that $z - x - z'$.
\end{prop}

\begin{proof}
This follows directly from the uniqueness of extension property.
\end{proof}

The point $z'$ from Proposition \ref{prop-antipode} is called the
\emph{antipode}  of $z$ in $S(x,r')$. We now define the antipodal map
for $B(x,r)$.

\begin{defn}[Antipodal Map] \label{def-antipode}
Suppose that $X$ is a  Busemann $G$-space and $B(x,r) \subset X,$ where
$0<r<\rho(x)$. Then the \emph{antipodal map $\Phi:B(x,r) \to B(x,r)$} is
defined so that $\Phi(x)=x$ and for each point $z \in S(x,r'),$ where $0<r'\leq r$,
$\Phi(z) = z'$, where $z'$ is  the unique antipode of $z$ in
$S(x,r')$.
\end{defn}

In the continuity arguments that follow, we shall use the following
extensively, to show that a map is continuous:  A map $f: X \to Y$
between compact metric spaces is continuous if and only if for every
point $x \in X$ and every
sequence $\{x_n\} \subset X$ such that $x_n
\to x$ and $f(x_n) \to y^*$,
one gets
$y^* = f(x)$.

\begin{prop}
The antipodal map $\Phi:B(x,r) \to B(x,r)$
from Definition~\ref{def-antipode} is a
homeomorphism.
\end{prop}

\begin{proof}
Suppose that $\{z_n\}  \subset B(x,r)$
is a sequence such
that $z_n \to z$
and $\Phi(z_n) \to z^* $.
Let $r' = d(x,z)$ and $r_n = d(x,z_n)$.
Then $r_n \to r'$.
Note that
$$d(z_n,x)+d(x,\Phi(z_n))
= d(z_n, \Phi(z_n))= 2r_n$$
which, by continuity of the
distance function, implies
that $$d(z,x)+d(x,z^*)
= d(z, z^*) = 2r'.$$
However, we also have
$$d(z,x)+d(x,\Phi(z))  = d(z,\Phi(z))= 2r'.$$ Moreover,
$$d(x,z^*) =
d(x,\Phi(z)) = r'.$$ By uniqueness of antipodes (cf. Proposition \ref{prop-antipode}), $z^* = \Phi(z)$.
Therefore, $\Phi$ is continuous.

Since $\Phi^{-1} = \Phi$, it follows that $\Phi$ is indeed a
homeomorphism.
\end{proof}

The following projection map will also be key in defining our
homogeneity homeomorphism.

\begin{defn}[Projection Map] \label{projection}
Suppose that $X$ is a Busemann $G$-space and $T,Z \subset X$ are compact star-like
sets with respect to $x \in  {int}(T) \cap {int}(Z)$. Define the
\emph{projection map $\psi: Z - \{x\} \to
\partial T$} such that for each
point $z \in Z - \{x\}$, $\psi(z)=t$, where $t$ is the unique point
of $\partial T$ so that one of $x-t-z$, $t=z$, or $x-z-t$ holds.  We
say that $\psi$ is \emph{centered at $x$}.
\end{defn}

The existence and uniqueness of $t$ in Definition
\ref{projection} easily follows from the definition of star-like set and inclusion
$x \in  {int}(T) \cap {int}(Z)$ .  The cases $x-t-z$, $t=z$, and $x-z-t$ correspond
the cases $d(x,t) < d(x,z)$, $d(x,t) = d(x,z)$, and $d(x,t)>
d(x,z)$, respectively.

\begin{prop}
The projection map $\psi: Z - \{x\} \to
\partial T$
from Definition~\ref{projection} is
continuous.  Moreover, the restriction map $\psi |_{\partial Z} :
\partial Z \to
\partial T$ is a homeomorphism.
\end{prop}

\begin{proof}
By hypothesis, $T,Z$ are starlike sets with respect
to $x$. For any $z \in \partial Z - \{ x \}$, $\psi(z)$ is the
unique point $t \in \partial T$ so that one of $x-t-z$, $t=z$, or
$x-z-t$ holds.

We shall now show the continuity of $\psi$. In particular, we shall
show that $\psi$ is continuous on the restriction to any compact set
$Z_\delta = \overline{Z - B(x,\delta)}$, where $B(x,\delta) \subset
int(Z)$ and $\delta>0$. Suppose that there is a sequence $\{z_n\} \subset Z_\delta$
such that $z_n \to z$ and $\psi(z_n) = t_n \to t^*$. By the
compactness of $T$, and hence $\partial T$, $t^* \in \partial T$.

For each $n$, one of $x-t_n-z_{n}$ or $x-z_{n}-t_n$ is the case. Let
$$L = \{ z_{n} \ | \ x-z_{n}-t_n \}$$ and $$M = \{ z_{n} \ | \
x-t_n-z_{n}\}.$$

If $L$ is finite, it follows from the continuity of the distance
function and the relation $d(x,t_n) + d(t_n,z_n) = d(x,z_n)$ for
large $n$,  that $x - t^* - z$ or $t^*=z$. If $M$ is finite, then
$d(x,z_n) + d(z_n,t_n) = d(x,t_n)$ for large $n$,  so that $x - z -
t^* $ or $z=t^*$. If neither $L$ nor $M$ is finite, then $z=t^*$
must be the case. However, $t$ is the unique point of $\partial T$
so that one of $x-t-z$, $t=z$, or $x-z-t$ holds. Hence $t^* = t$.
Therefore $\psi$ is continuous.

Note that both $\psi |_{\partial Z}$ and $\psi|_{\partial Z}^{-1}$
are well-defined and 1-1 by our choice of $Z$ and $T$. Since $\psi$
is continuous, $\psi|_{\partial Z}$ is also continuous. The map
$\psi|_{\partial Z}^{-1}$ is the restriction of the projection map
$\phi: T-{x} \to \partial Z$ to $\partial T$.   Thus the continuity
of $\psi |_{\partial Z}^{-1}$ also follows from a similar argument
as above. Therefore $\psi|_{\partial Z}$ is a homeomorphism.
\end{proof}

\begin{thm}
In a locally $G$-homogeneous Busemann $G$-space $X$, every two
spheres $S(x,r(x))$ and $S(y,r(y))$, of radii $0<r(x)<\rho(x)$ and
$0<r(y)<\rho(y),$ respectively are homeomorphic.
\end{thm}

\begin{proof}  It suffices to show that for every $x\in X$ there is a $\delta > 0$
such that the result is true for each $y \in B(x,\delta)$.  Let
$B(x,\epsilon)$ be the ball promised by the definition of locally
$G$-homogeneous and $\delta
>0$ ($\delta < \epsilon$) be the value promised by the definition of stably starlike. Then for any $y
\in B(x,\delta)$ we have the desired homeomorphism to be the
composition of homeomorphisms
$$S(x,r(x)) \overset{\psi_1}{\to} S(x,\epsilon)
\overset{\psi_2}{\to} S(y,r(y))$$  where $\psi_1$ is the projection
map centered at $x$ and $\psi_2$ is the projection map centered at
$y$.

\end{proof}

We are now ready to prove a weaker version of our first main theorem.

\begin{thm} \label{Main Thm}
Suppose that $X$ is a locally $G$-homogeneous Busemann $G$-space. Then any metric sphere
$S(x,r), 0<r<\rho(x),$ is topologically homogeneous.
\end{thm}

\begin{proof}
It suffices to show that $S(x,\varepsilon)$ is homogeneous for
sufficiently close points.  Without loss of generality, choose $y$
sufficiently close to $z$ so that $B(x,\varepsilon)$ is starlike
with respect to the midpoint $m$ of $\overline{y'z}$. Let $\gamma=
\frac 12 d(y,y')$. The desired map sequence provides a homeomorphism
taking $y$ to $z$, where $\Phi_i$ are antipodal maps and $\psi$ is a
projection map centered at $m$.

$$S(x,\epsilon) \overset{\Phi_1}{\to}
S(x,\epsilon)\overset{\psi}{\to} S_\gamma(m)\overset{\Phi_2}{\to}
S_\gamma(m)\overset{\psi^{-1}}{\to} S(x,\epsilon).$$
\end{proof}

\section{Strong Homogeneity}

In this section we shall show that small metric spheres in
locally $G$-homogeneous Busemann $G$-spaces $X$ are in fact,
strongly homogeneous. We shall call a sphere $S(x,r)\subset X$
\textit{sufficiently small} if $r< \rho(x)$ 
(cf. Lemma~\ref{rho}).

\begin{defn}  Let $X$ be a Busemann $G$-space.
Then
$\Omega \subset X$ is said to be
a \emph{fundamental
region in $X$} provided
that:
\begin{enumerate}
\item $\Omega$ is an open region with compact closure; and
\item For any closed metric ball $B(x,r) \subset \Omega,$
$r<\rho(x).$
\end{enumerate}
\end{defn}

Each point in a Busemann $G$-space $X$ is contained in
a fundamental region. For example, one can easily prove that
every open ball $U(x,r)\subset X,$ where $0<r<\rho(x)$ is a
fundamental region.

We shall begin with some continuity theorems.

\begin{thm}
Suppose that $X$ is a  Busemann $G$-space and $\Omega \subset X$ is
a fundamental region.  Let $$\hat{\Omega}= \{(a,r,x) \in X \times
\mathbb{R}\times X\ | \  x\in B(a,r) \subset \Omega\}$$ and
$\Phi[a,r]: B(a,r) \to B(a,r)$ is an antipodal map.  Then the map
$$f: \hat{\Omega} \to X;\  f(a,r,x): = \Phi[a,r](x)$$  is continuous.
\end{thm}

\begin{proof}
Let ${\Omega^*}$
be a compact subset of
$\Omega$. Let
$$
\hat{\Omega}^*= \{(a,r,x) \in X \times \mathbb{R}\times X\ | \  B(a,r)
\subset \Omega^*\}.
$$
It suffices to show that $f$ is continuous on
$\hat{\Omega}^*$.

Suppose there is a sequence $(a_n,r_n,x_n) \to (a,r,x)$ in
$\hat{\Omega}^*$ and $f(a_n,r_n,x_n) \to x^*$. Without loss of
generality we may choose $r'_n$ so that $r_n \leq r'_n$ and $B(a,r)
\subset B(a_n,r'_n) \subset \Omega$ (this can be accomplished if the
sequence  $\{a_n\}$ is modified to contain only points very close to
$a$). Define $x'_n = f(a_n,r'_n,x_n)$.
Note that
$$d(x'_n,a_n)+d(a_n,x) = d(x'_n,x_n) \quad \text{ and } \quad d(x'_n,a_n)=d(a_n,x_n).$$ By
continuity of the distance function  $$d(x^*,a)+d(a,x) = d(x^*,x)
\quad \text{ and } \quad d(x^*,a)=d(a,x).$$  However $x'$, the
antipode of $x$ in $B(a,r)$, satisfies these same relations in place
of $x^*$. It then follows from uniqueness of the antipode that $x^*
= x'$. Therefore $f$ is continuous.
\end{proof}

\begin{thm}
Suppose that $X$ is a Busemann $G$-space, $\Omega \subset X$ is a
fundamental region, $B(x,\rho) \subset \Omega$ is stably starlike at
$x$ and $\delta>0$ is a radius promised in the definition of stably
starlike (cf. Definition \ref{stably starlike}).  Let
$$
\hat{\Omega} = \{(a,r,y) \ | \ a \in B(x,\delta) \subset U(a,r)
\subset B(a,r) \subset \Omega, y\in \overline{B(x,\rho)-
B(x,\delta)}\}.
$$
Then the map $g: \hat{\Omega} \to X; g(a,r,y) = \psi[a,r](y),$ where
$\psi[a,r]: \overline{B(x,\rho)- B(x,\delta)} \to S(a,r)$ is the
projection map centered at $a$, is continuous.
\end{thm}

\begin{proof}
Let  $\Omega^*$ be a compact subset of $\Omega$.  Define

$$\hat{\Omega}^* = \{(a,r,y) \in \Omega \ | \  y\in B(a,r) \subset \Omega^*
\}.
$$  It suffices to show that $g$ is continuous on $\hat{\Omega}^*.$

Suppose there is a
sequence $(a_n,r_n,y_n)\in \Omega^*$ such that $(a_n,r_n,y_n)
\to (a,r,y)$ and $g(a_n,r_n,y_n) \to y^*$.  Let $y_n' = g(a_n,r_n,y_n)$.
 Note that $y^* \in S(a,r)$.  Also, precisely
one of $a_n - y_n' - y_n$, $y_n'=y_n$, or $a_n - y_n - y_n'$ holds for each
$n$. By continuity of the distance function, one of $a - y^* -
y$, $y^*=y$, or $a-y-y^*$ holds. However, $y'=g(a,r,y)$ is a
point of $S(a,r)$ that also satisfies this condition when replaced
with $y^*$.  Thus $y^* = y'$.  Therefore $g$ is indeed continuous.

\end{proof}

\begin{thm}
Suppose that $X$ is a Busemann $G$-space, $\Omega \subset X$ is a
fundamental region and $B(x,\epsilon) \subset \Omega$. Let $y$ be a
fixed point in $S(x,\epsilon)$ and let $y'$ denote the antipode of
$y$ in $S(x,\epsilon)$.  Then the midpoint map $$\Gamma:
S(x,\epsilon) \to X$$ such that $\Gamma(z) = m,$ where $m$ is the
midpoint of $\overline{zy'},$ is continuous.
\end{thm}

\begin{proof}  Suppose $\{ z_n \} \subset S(x,\epsilon)$
is a sequence
such that $z_n
\to z$ and $\Gamma(z_n) = m_n  \to  m^*$.  Then $z _n - m_n - y'$
and $d(z_n, m_n) = d(m_n, y')$. By the continuity of the distance
function, $z- m^* - y'$ and $d(z,m^*) = d(m^*,y')$.
However,
$m$ also
satisfies these relations in place of $m^*$.
By uniqueness of joins,
$m = m^*$.  Thus $\Gamma$ is indeed
continuous.

\end{proof}

We are now ready to prove the strong form of the first main theorem.

\begin{proof}[Proof of Theorem~\ref{main}]
Let $y,z \in S(x,\epsilon)$ and $\alpha: I \to S(x,\epsilon)$ be a
path from  $y$ to $z$.  Let $y'$ be the antipode of $y$, $m(t)$ the
midpoint of $\overline{\alpha(t) y'}$, and $\gamma(t) = \frac 12
d(\alpha(t), y')$. The isotopy $H:S_\epsilon(x) \times I \to
S_\epsilon(x)$ is given by $H_t$ which is the composition of
homeomorphisms

$$S_\epsilon(x) \overset{\Phi[x,\epsilon]}{\to}
S_\epsilon(x)\overset{\psi[m(t),\gamma(t)]}{\to}
S_{\gamma(t)}(m(t))\overset{\Phi[m(t),\gamma(t)]}{\to}
S_{\gamma(t)}(m(t))\overset{\psi[m(t),\gamma(t)]^{-1}}{\to}
S_\epsilon(x).$$  Continuity of $H_t$ in the variable $t$ follows
from the propositions above. Thus $H$ is the desired isotopy.
\end{proof}

\section{Uniformly locally $G$-homogeneous Busemann $G$-spaces}

\begin{defn}
\label{unif} We say that a metric space $(X,d)$ is \textit{uniformly
locally $G$-homogeneous on a set} $C\subset X$ if there are  numbers
$\delta,\varepsilon_1,\varepsilon_2$ such that
\begin{itemize}
\item $0<\delta\leq \varepsilon_1<\varepsilon_2,$

\item For every $c\in C$ and every  $\varepsilon\in (\varepsilon_1,\varepsilon_2)$ the closed ball $B(c,\varepsilon)$
is starlike with respect to every point $x$ in open ball $U(c,\delta),$
\end{itemize}
We say that a metric space $(X,d)$ is \textit{uniformly locally
$G$-homogeneous on an orbal subset} $C\subset X$ if additionally
\begin{itemize}
\item $C$ contains a ball $B(c_0,r)$ such that $\varepsilon_2<r.$
\end{itemize}
\end{defn}

The following is our second main theorem (cf. Theorem 1.2 from the Introduction).

\begin{thm}
\label{fd} If a Busemann $G$-space is uniformly locally
$G$-homogeneous on an orbal subset $C$, then it has  finite
topological dimension.
\end{thm}

\begin{proof}
Assume the setup given be Definition \ref{unif}.  Then we can find
numbers $\varepsilon_1',\varepsilon_2'$ such that
\begin{equation}
\label{eps}
\varepsilon_1< \varepsilon_1'< \varepsilon_2'<\varepsilon_2 \quad\mbox{and}\quad \frac{\varepsilon_2'-\varepsilon_1'}{2}+\varepsilon_2'<\varepsilon_2.
\end{equation}
Let us choose
numbers
\begin{equation}
\label{numbers}
r_1=\frac{1}{2}\min\left(\varepsilon_2'-\varepsilon_1',\delta\right) \quad\mbox{and arbitrary}\quad \varepsilon_0: 0< \varepsilon_0< \min(\delta,\varepsilon_2-\varepsilon_2',\varepsilon_1'-\varepsilon_1).
\end{equation}
Consider the set
$$D=\{(x,z)\in B(c_0,r_1)\times B(c_0,r)| d(x,z)=\varepsilon_2'\}$$
and define a metric $d_1$ on $D$ by the formula
\begin{equation}
\label{d}
d_1((x,z),(x',z'))=\max(d(x,x'),d(z,z')).
\end{equation}
Evidently the metric space $(D,d_1)$ is compact. Thus there  is a
finite $\varepsilon_0$-net
$$\{(x_1,z_1),\dots, (x_m,z_m)\}$$ in $(D,d_1).$
Define a (continuous) map $f:B(c_0,r_1)\rightarrow \mathbb{R}^m$ by the
formula
$$f(y)=(d(y,z_1),\dots, d(y,z_m)).$$
We state that $f$ is a topological embedding. For this it is enough
to show that $f(x)\neq f(y)$ if $x,y\in B(c_0,r_1)$ and $x\neq y.$
Indeed, it follows from the Triangle inequality and the formula
(\ref{numbers}) that $d(x,y)\leq
\varepsilon_2'-\varepsilon_1'<\varepsilon_2'.$ Using once more
Definition \ref{unif} and equations (\ref{eps}) and (\ref{numbers}),
we see that there is unique extension of the segment $\overline{xy}$
to a segment $\overline{xz}$ of length $\varepsilon_2',$ and this
segment lies in $B(c_0,r).$ Clearly, $(x,z)\in D.$ By construction,
there is an index $i\in\{1,\dots,m\}$ such that
$d_1((x_i,z_i),(x,z))<\varepsilon_0.$ We claim that $d(x,z_i)\neq
d(y,z_i)$ (and so $f(x)\neq f(y)$). Otherwise, using the Triangle
inequality and equations (\ref{eps}), (\ref{numbers}), and
(\ref{d}), we see that
\begin{equation}
\label{see}
\varepsilon:=d(z_i,x)=d(z_i,y)\in (\varepsilon_1,\varepsilon_2).
\end{equation}
On the other hand, $z-y-x$ and $d(z_i,z)<\delta,$ and the equation
(\ref{see}) contradicts the statement that the closed ball
$B(z_i,\varepsilon)$ must be starlike with respect to the point $z.$
So $f$ is one-to-one on $B(c_0,r_1)$ and the topological  dimension
of $B(c_0,r_1)$ is less than or equal to $m.$ Now Corollary
\ref{homo} implies that the topological dimension of $(X,d)$ is less
than or equal to $m.$
\end{proof}

\begin{remark}
\label{generalization} As we said in the Introduction, it was proved
in \cite{Berestovskii1} that a Busemann $G$-space $X$ is
finite-dimensional if $X$ has small metrically convex balls near
some of its points. It is known that this implies that $X$  also has
small metrically strongly convex balls near the same points
\cite{Busemann3}. This fact together with Remark \ref{convexity}
implies that a Busemann $G$-space, which has small convex balls near
some point, is also uniformly locally $G$-homogeneous on an orbal
subset. So Theorem \ref{fd} generalizes the result from
\cite{Berestovskii1} mentioned above.
\end{remark}

\section{example}

In this section we prove Theorem~\ref{ex}, i.e. we present an
example of a Busemann $G$-space that is uniformly locally
$G$-homogeneous on an orbal subset and locally $G$-homogeneous, but
has no convex metric ball of positive radius.

In 1999 Gribanova \cite{Gribanova} found all inner metrics on the
upper half plane which are invariant under the action of the group
$$\Gamma: x'=\alpha x + \beta, y'=\alpha y, \alpha > 0, -\infty<\beta<+\infty,$$
as well as their geodesics. It follows from
\cite{Berestovskii5}
that every such metric must be Finslerian.
Thus it is easy to see that the corresponding line element must have a
form $ds = y^{-1} F(dx, dy)$ with a fixed norm $F.$
Gribanova completely classified all quasihyperbolic
geometries determined by the above line element
(i.e. Busemann $G$-spaces)
depending on the properties of $F$.

In particular, the following theorem was proven:

\begin{thm}
\label{quasi}
The line element of a quasihyperbolic plane can be written in the
form $$ds = y^{-1} F(dx, dy);$$ moreover, the function $F(u_1, u_2)$
is defined for all $u_1$ and $u_2$ and satisfies the following
conditions
\begin{enumerate}
\item $F(u_1, u_2) > 0$ for $(u_1,u_2) \ne (0,0)$;
\item $F(ku_1,ku_2) = | k | F(u_1,u_2)$ for every real $k$;
\item $F$ is convex;
\item $F$ is differentiable everywhere except for $(0,0)$;
\item The tangents of the curve $F(u_1, u_2) = 1$, parallel to the
straight line $u_2=0$, touch this curve at a unique point.
\end{enumerate}
Conversely, each line element of the form $$ds = y^{-1} F(dx, dy)$$
with $F$ possessing these properties determines a quasihyperbolic
geometry.

Define the function $F^*(x,y)$ by $$F^* = \underset{F(u_1,u_2) \leq
1}{\max } (xu_2-yu_1).$$  Geodesics of the quasihyperbolic plane
with the line element $ds = y^{-1} F(dx, dy)$ are the intersections
of the half-plane $y> 0$ with the curves $F^*(x-a,y) = k, k > 0,
-\infty < a < \infty$, and with the tangents of these curves at
their intersection points with the $x$-axis.  For two distinct
points of a quasihyperbolic plane, there is exactly one geodesic
passing through them.

\end{thm}

\begin{remark}
Geometric meaning of geodesics is that they are solutions of
isoperimetric problem for two-dimensional normed vector space with
the norm $F$ \cite{Busemann5}. Here, as well as below,  a
prescription is given how to construct them. However, Busemann also
said in \cite{Busemann6} (cf. p. 82) that spaces, defined in this
manner by two norms $F_1$ and $F_2,$ are  isometric if and only if
there is a linear transformation $l$ of $\mathbb{R}^2$ such that
$F_2=F_1\circ l.$ This statement is not true because the usual
Euclidean norm $F(u,v)=\sqrt{u^2+v^2}$ gives a hyperbolic plane of
curvature $-1,$ while the norm $kF, k>0,$ gives a hyperbolic plane
of curvature $\frac{-1}{k^2}.$ 
\end{remark}

Gribanova also proved a theorem which can
equivalently
be stated as follows.

\begin{thm}
\label{nonconvex}
Suppose also that both tangent lines to the curve (so-called indicatrix)
 $C= \{ (x,y) \in \mathbb{R}^2 \ | \ F(x,y)= 1 \}$
at intersection points of this curve with $x$-axis
have nontrivial joint segments with the curve $C.$
Then no closed ball $B(p,r)$ of the
quasihyperbolic plane with the line element $ds = y^{-1} F(dx, dy)$
is geodesically convex for any $r>0$.
\end{thm}

\subsection{Stadium space norm}

Let us consider a quasihyperbolic plane $X$ which we shall call the
``Stadium Space''. It is defined by the set "Stadium" which consists
of squares with side length $2$ together with semidisks on the top
and the bottom with radius $1$, as pictured in Figure \ref{Stadium}.

 \begin{figure}
     \begin{center}
\epsfig{file=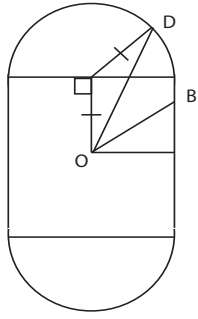, width=1.0 in, height=1.62in}
 \caption{The Stadium} \label{Stadium}
  \end{center}
\end{figure}

The Stadium defines a norm $F=\|\cdot\|$ on $\mathbb{R}^2$, if we
assume 
that its boundary curve $C$ is a unit circle. It is clear
that the norm $F$ satisfies all hypotheses of Theorems \ref{quasi} and
\ref{nonconvex}. The norm of a vector $\mathbf{v}$ in the direction with the
angle $\psi$
measured from the positive $x$-axis is equal to $$\| \mathbf{v} \| =
l(\psi) |\mathbf{v}|,$$ where $|\mathbf{v}|$ is the usual Euclidean
norm and
\begin{equation}
\label{length}
l(\psi) = \left\{
                    \begin{array}{ll}
                         |\cos \psi | , & \text{if} -\frac{\pi}{4} \leq \psi \leq \frac{\pi}{4} \text{ or } \frac{3\pi}{4} \leq \psi \leq \frac{5\pi}{4} \\
                         \frac{1}{2|\sin \psi | }, & \text{if} \frac{\pi}{4} \leq \psi \leq \frac{3\pi}{4} \text{ or } -\frac{3\pi}{4} \leq \psi \leq -\frac{\pi}{4}
                       \end{array}
                     \right.
\end{equation}

\noindent To see this, observe in Figure \ref{Stadium} that
$l(\psi)$ is $\frac{1}{OB}$  in the first case and $\frac{1}{OD}$
in the second case.

\subsection{Geodesics }

Geodesics in the metric geometry are determined by the dual curve to
the Stadium space (cf. Figure \ref{dual}).
$$C_1 = \{(x,y) \in \mathbb{R}^2 \ | \
\underset{(u_1,u_2) \in \mathcal{C}}{\max} (xu_1 + yu_2) = 1 \}.$$
In particular, if $C_1$ is rotated by $\frac{\pi}{2}$ to obtain
$C_2$ (cf. Figure \ref{geo}), then the geodesics are the portions of
the  vertical lines and curves of the form
$$C_2(\lambda,x_0):=\lambda C_2+(x_0,0); \quad \lambda>0, \quad x_0\in
\mathbb{R}$$  contained in $\mathbb{R}^2_+=\{(x,y)\in \mathbb{R}^2:
y>0\}.$

 \begin{figure}
    \begin{center}
\epsfig{file=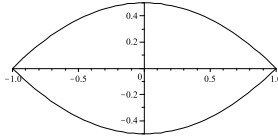, width=1.5in, height=1.5in}
\vspace{-0.6 in}
        \caption{The dual curve ${C}_1$} \label{dual}
    \end{center}
\end{figure}

\begin{figure}
    \begin{center}
\epsfig{file=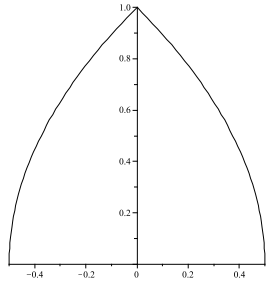, width=1.5in, height=1.5in}
        \caption{The standard geodesic ${C}_2$} \label{geo}
    \end{center}
\end{figure}

The dual curve $C_1$ to the boundary of the Stadium space  is
defined in polar coordinates, $[\tau, \phi]$ by the following
function $\tau_1(\phi).$

\begin{align*}
\tau_1(\phi) = \frac{1}{OE} = \frac{1}{1+ | \sin \phi |},
\end{align*}
where $E$ is the orthogonal projection of the point $O$ onto the
tangent line to the curve $C$ at a point $D\in C$ (cf. Figure
\ref{Stadium3}.)

 \begin{figure}
     \begin{center}
\epsfig{file=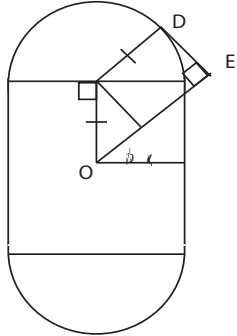, width=1.2 in, height=1.75in}
 \caption{Defining the dual curve} \label{Stadium3}
  \end{center}
\end{figure}

\noindent Rotating by $\frac{\pi}{2}$  we get the defining function
$\tau_2$ for $C_2$ as

\begin{align*}
\tau_2 = \tau_1\left(\phi-\frac{\pi}{2} \right) =  \frac{1}{1+ | \sin (\phi-\frac{\pi}{2}) | }
    = \frac{1}{1+ | \cos \phi | }.
\end{align*}

We shall omit calculations for the left part of the curve $C_2$,
using later the symmetry of curves $C$ and $C_2$ relative to
$y$-axis. Also we consider only the upper  halves of curves
$C_2(\lambda,x).$  So we get the equation
\begin{equation}
\label{right} \tau_2 = \frac{1}{1+ \cos \phi  }, \quad 0<\phi\leq
\pi/2.
\end{equation}
It is known that this is part of parabola. Setting $\phi=0$ and
$\phi=\pi/2,$ we see that the right side of $C_2$ has equation
$$x=\frac{1-y^2}{2}.$$ Hence, the right side of $\lambda C_2$ has
equation
$$x=\frac{\lambda^2-y^2}{2\lambda}.$$ So the entire curve $\lambda C_2$ is
\begin{equation}
\label{c2} x=\pm \frac{\lambda^2-y^2}{2\lambda} , \quad 0\leq
|y|\leq \lambda.
\end{equation}
Since the metric space, which we consider, is homogeneous, we can
study only the circles of this metric with the center at the point
$(0,1).$ On the curve $\lambda C_2$, if $y=1$ we get
$x_1=\frac{\lambda^2-1}{2\lambda}.$ So the shifted curve $\lambda
C_2 + (x_0(\lambda),0)$ passing through $(0,1)$ has shifting term
\begin{equation}
\label{xlambda}
x_0(\lambda)=-x_1=\frac{1}{2}\left(\frac{1}{\lambda}-\lambda\right)
\end{equation}
So we get the equation of the right side of $ \lambda C_2 +
x_0(\lambda)$ to be
\begin{equation}
\label{para} x=\frac{\lambda^2-y^2}{2\lambda}+x_0(\lambda)=
\frac{1-y^2}{2\lambda}, \quad  0<y\leq \lambda, \quad 1\leq \lambda.
\end{equation}

 \begin{figure} \label{geoshift}
    \begin{center}
\epsfig{file=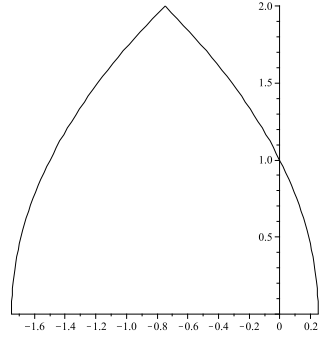, width=1.8in, height=1.8in}
        \caption{Shifted geodesic $\lambda C_2 + (x_0(\lambda),0).$}
    \end{center}
\end{figure}

\subsection{Distance formulas}

\noindent The tangent vector for the right side of $ \lambda C_2 +
x_0(\lambda)$ has direction  vector
$$\left \langle \frac{dx}{dy},1\right \rangle =
\left \langle -\frac{y}{\lambda},1\right \rangle, \quad\mbox{where}\quad 0<\frac{y}{\lambda}\leq 1.$$
Then the angle $\psi$ of this direction changes  between $\pi/2$ and
$3\pi/4.$ So we need to use only the second formula in
(\ref{length}). Here
$$\frac{1}{2|\sin \psi|}=\frac{\sqrt{1+(\frac{y}{\lambda})^2}}{2}.$$
The line element on $\mathbb{R}^2_+=\{(x,y)\in \mathbb{R}^2: y>0\}$
is $ds=\frac{1}{y}\|\cdot\|.$ Then the length of a geodesic between
two points $(x_1, y_1)$ and $(x_2,y_2)$ on the right side of $
\lambda C_2 + x_0(\lambda)$ where $0<y_1<y_2\leq \lambda$ is
$$l(y_1,y_2)=\int_{y_1}^{y_2}\frac{\sqrt{1+(\frac{y}{\lambda})^2}}{2}\frac{\sqrt{1+(\frac{y}{\lambda})^2}}{y}dy $$
$$=\frac{1}{2}\int_{y_1}^{y_2}\frac{1+(\frac{y}{\lambda})^2}{\frac{y}{\lambda}}d\left(\frac{y}{\lambda}\right) $$
$$=\frac{1}{2}\int_{\frac{y_1}{\lambda}}^{\frac{y_2}{\lambda}}\frac{1+z^2}{z}dz
=\frac{1}{2}\int_{\frac{y_1}{\lambda}}^{\frac{y_2}{\lambda}}(z+\frac{1}{z})dz
$$
$$=\frac{1}{2}\left(\ln z+\frac{z^2}{2}\right)\bigg|_{\frac{y_1}{\lambda}}^{\frac{y_2}{\lambda}}=
\frac{1}{2}\ln \frac{y_2}{y_1}+\frac{1}{4\lambda^2}(y_2^2-y_1^2).$$
In the subcases $y_2=1>y_1=y$ or $y=y_2>y_1=1$ we get respectively
\begin{equation}
\label{lessy} l(y)=\frac{1}{4}\left(\frac{1-y^2}{\lambda^2}-\ln
y^2\right) \quad\mbox{for}\quad y<1 \quad  \mbox{(on the right side
of the curve)}
\end{equation}
or
\begin{equation}
\label{morey} l(y)=\frac{1}{4}\left(\frac{y^2-1}{\lambda^2}+\ln
y^2\right) \quad\mbox{for}\quad y>1 \quad  \mbox{(on the right side
of the curve)}.
\end{equation}

\noindent Applying the last formula to the point of maximal height
$(x_0(\lambda),\lambda),$ where $\lambda>1,$ we get
$$l(\lambda)=\frac{1}{4}\left(\frac{\lambda^2-1}{\lambda^2}+\ln\lambda^2\right)=
\frac{1}{4}\left(1-\frac{1}{\lambda^2}+\ln\lambda^2\right),$$
so
\begin{equation}
\label{la}
l(\lambda)=\frac{1}{4}\left(1-\frac{1}{\lambda^2}+\ln\lambda^2\right).
\end{equation}

\subsection{Metric spheres}

Now we shall find a form of the sphere $S_K:=S((0,1),K)$ with radius
$K>0$ and   center $(0,1),$ using only the right side of curves
$\lambda C_2 + x_0(\lambda)$. Then we can apply the symmetry of the
geodesic relative to the line $x=x_0$.

We shall have three cases when the geodesic radii to the point on
$S_K$ is nonvertical  (cf. Figure \ref{georadii}).

 \begin{figure}
    \begin{center}
\epsfig{file=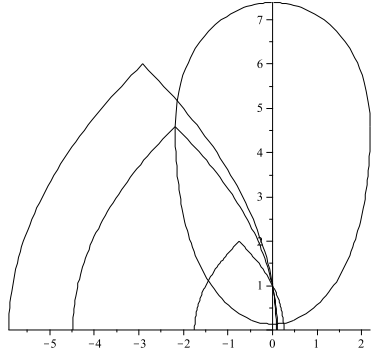, width=2.0in, height=2.0in}
        \caption{A unit sphere in the Stadium space with sample geodesic
        radii} \label{georadii}
    \end{center}
\end{figure}

1) If $(x,y)\in S_K,\ x>x_0,\ y<1,$ then we have by formula
(\ref{lessy}) that
$$K=l(y)=\frac{1}{4}\left(\frac{1-y^2}{\lambda^2}-\ln y^2\right).$$

2) If $(x,y)\in S_K,\  x \geq x_0, \ y>1$, then $K\leq l(\lambda)$ and
we have by formula (\ref{morey}) that
$$K=l(y)=\frac{1}{4}\left(\frac{y^2-1}{\lambda^2}+\ln y^2\right).$$

3) Consider now the case when $(x,y)\in S_K,\ x<x_0.$  Note that in
this case $K>l(\lambda).$ Using the symmetry of the geodesic with
respect to the line $x=x_0$ we have
\begin{align*}
K= l(\lambda) + d((x_0(\lambda),\lambda),(x,y)) =  \frac{1}{4}\left(1-\frac{1}{\lambda^2}+\ln\lambda^2\right) +  \frac{1}{2}\ln\frac{\lambda}{y} +\frac{\lambda^2-y^2}{4\lambda^2},
\end{align*}
which is equivalent to the equation
\begin{equation}
\label{middle}
\frac{y^2+1}{4\lambda^2}+\frac{1}{2}\ln y-\ln \lambda=\frac{1}{2}-K.
\end{equation}
In this case

\begin{align*}
x=- \frac{\lambda^2-y^2}{2\lambda} + x_0(\lambda)= - \frac{\lambda^2-y^2}{2\lambda} +
\frac{1}{2}\left(\frac{1}{\lambda}-\lambda\right) = -\lambda+\frac{y^2+1}{2\lambda}.
\end{align*}

\noindent Hence
\begin{equation}
\label{xafter} x=-\lambda+\frac{y^2+1}{2\lambda}.
\end{equation}

In the case of a vertical geodesic radii, the points on $S_K$ are
easily evaluated from solving the integral equation

$$K=\left | \int_{1}^{y_0}\frac{1}{2y} dy \right | $$
for $y_0$.  It follows that the boundary points along $x=0$ are
$(0,e^{\pm{2K}})$. Note that we can also get this by taking limits
as $\lambda\rightarrow +\infty$ in formulas (\ref{lessy}) and
(\ref{morey})  where $l(y)=K$.

\subsection{Tangents to spheres}

Next we shall find tangents to the sphere $S_K,$ using the 
equations above
and considering only its right part. This part in turn,
consists of 3 curves: 
the bottom right curve $B_1,$  
the side right curve $B_2,$ and 
the top right curve $B_3$
(cf. Figure \ref{unitball}). The
joint point of curves $B_1$ and $B_2$ is defined by the
equality
$\lambda=1,$ while the joint point of curves $B_2$ and $B_3$ is
defined by the
equality $y_0=\lambda_0,$ where $l(\lambda_0)=K$ (cf.
equation (\ref{la})).

 \begin{figure}
    \begin{center}
\epsfig{file=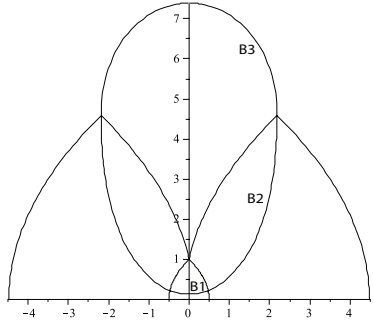, width=2.0in, height=2.0in}
 \caption{The boundary curves of a unit ball} \label{unitball}
    \end{center}
\end{figure}

Equations of the bottom right curve $B_1$ are

$$x=\frac{1-y^2}{2\lambda},$$

$$\frac{1-y^2}{4\lambda^2}-\frac 12 \ln y =K.$$

\noindent Differentiating these equations and using once more the
first one of them, we get the following system

$$2\lambda y \frac{dy}{dx} + (1-y^2) \frac{d \lambda}{dx} =
-2\lambda^2,$$

$$\lambda (y^2 + \lambda^2) \frac{dy}{dx} + y (1-y^2) \frac{d
\lambda}{dx} = 0.$$

\noindent Solving for $\frac{dy}{dx}$ and $\frac{d\lambda}{dx}$ we get
by the
Cramer rule 

\begin{equation}
\label{drb}
\frac{dy}{dx}=\frac{2y\lambda}{\lambda^2-y^2}>0,\quad
\frac{d\lambda}{dx}=\frac{2\lambda^2(y^2+\lambda^2)}{(1-y^2)(y^2-\lambda^2)}<0.
\end{equation}

\indent In the
other two cases we need to multiply the
equations
(\ref{xafter}) and (\ref{para}) by $-1.$
Equations of the side right curve $B_2$ are

$$x=\lambda -\frac{1+y^2}{2\lambda},$$

$$\frac{y^2+1}{4\lambda^2}+\frac 12 \ln y-\ln \lambda =\frac 12 -K.$$

\noindent We differentiate these equations to get

$$ -2\lambda y \frac{dy}{dx} + (2\lambda^2 + 1 +  y^2)
\frac{d\lambda}{dx} = 2\lambda^2,$$

$$\lambda (y^2 + \lambda^2) \frac{dy}{dx} - y(  2\lambda^2 + 1 + y^2)
 \frac{d \lambda}{dx} = 0.$$

\noindent Solving for $\frac{dy}{dx}$ and $\frac{d\lambda}{dx},$
using the Cramer rule, we get the following

\begin{equation}
\label{drs}
\frac{dy}{dx}=\frac{2y\lambda}{\lambda^2-y^2}>0, \quad
\frac{d\lambda}{dx}=\frac{2\lambda^2(y^2+\lambda^2)}{(2\lambda^2+1+y^2)(\lambda^2-y^2)}>0.
\end{equation}

Equations of the top right curve $B_3$ are

$$x=\frac{y^2-1}{2\lambda},$$

$$\frac{y^2-1}{4\lambda^2}+\frac 12 \ln {y} =K$$

\noindent By differentiating these equations we get

$$2\lambda y \frac{dy}{dx} + (1-y^2) \frac{d \lambda}{dx} =
2\lambda^2,$$

$$\lambda (y^2 + \lambda^2) \frac{dy}{dx} + y (1-y^2) \frac{d
\lambda}{dx} = 0.$$

\noindent Solving for $\frac{dy}{dx}$ and $\frac{d\lambda}{dx}$
using the Cramer rule, we get the following

\begin{equation}
\label{drt}
\frac{dy}{dx}=\frac{2y\lambda}{y^2-\lambda^2}<0, \quad
\frac{d\lambda}{dx}=\frac{2\lambda^2(y^2+\lambda^2)}{(1-y^2)(\lambda^2-y^2)}<0.
\end{equation}

We see from equations (\ref{drb}) and (\ref{drs}) that
$\frac{dy}{dx}$  is continuous on $B_1 \cup B_2$ 
except at the top
point of $B_2,$ where $B_2$ meets $B_3$ and $y=\lambda,$ so the
slopes of both curves $B_2$ and $B_3$ approach infinity and have
vertical tangents at their joint point. Also $\frac{dy}{dx} \to 0$
as $\lambda \to \pm \infty$ (or $x \to 0$). All these assertions
imply that the curve $S_K$ is smooth.

\subsection{Convexity properties}

To see that the ball $B((0,1),K)$ is not convex, it is enough to
take a geodesic $\lambda C_2\cap \mathbb{R}^2_+$ with a number
$\lambda,$ which is a little more than $e^{2K}.$

Using equations (\ref{drb}), (\ref{drs}), and (\ref{drt}), we can
show after a little tedious calculations that $\frac{d^2y}{dx^2}>0$
at interior points of $B_1$ and $B_2$ and $\frac{d^2y}{dx^2}<0$ at
interior points of $B_3.$ This implies that the curve $S_K$ is
strongly convex in affine sense.

\subsection{Uniform local G-homogeneity}  In order to get this
result, we look carefully at the geometry of the Stadium space.

Right-sided tangent vectors to every geodesic may have only
directions with angles in intervals $(\frac{\pi}{4},\frac{\pi}{2})$
or $( -\frac{\pi}{2},-\frac{\pi}{4} )$, with directions
$\pm\frac{\pi}{4}$ only at its top point, where $y=\lambda.$ This
implies that a geodesic with origin inside $B((0,1),K)$ can
intersect $B_1\cup B_2$ at most once. Thus geodesics with origin
inside $B((0,1),K)$ and parameters $\lambda, 1\leq \lambda \leq \lambda_0$ can
intersect the right side of $S_K$ at most once.

Further we shall consider without any mention only geodesics which
intersect the set $U ((0,1),K)\cap \{(x,y)\in \mathbb{R}^2| y=1.\}$
It is clear that the width of $S_K$ is equal to
$2|x(\lambda_0)|=\lambda_0-\frac{1}{\lambda_0}.$ Now one can easily
see that
$$2|x(\lambda)|=\lambda-\frac{1}{\lambda}\geq 2(\lambda_0-\frac{1}{\lambda_0})\quad\mbox{if}\quad \lambda\geq 2\lambda_0$$
and a geodesic with parameter $\lambda\geq 2\lambda_0$ can intersect
the right side of $S_K$ at most once. So we need to consider only
geodesics  with parameters $\lambda, \lambda_0<\lambda<2\lambda_0.$
It follows from equation (\ref{c2}) that right-side derivatives on
any geodesic with such parameter at any point $(x,y),$ where
$y_0=\lambda_0\leq y\leq \lambda,$ satisfy condition
$$\frac{dy}{dx}=\pm\frac{\lambda}{y}\geq -\frac{\lambda}{y}> -\frac{2\lambda_0}{\lambda_0}=-2.$$
This implies that geodesics with such parameters can intersect the
right side of $S_K$ at  least twice only at points with $y>y_1,$
where $\frac{dy}{dx}(y_1)=-2$ for derivative along $S_K.$
We can deduce from equations (\ref{drt})
that $y_1=\frac{\sqrt{5}-1}{2}\lambda_1,$ where
$$\frac{y_1^2-1}{4\lambda_1^2}+\frac 12 \ln {y_1} =K.$$
So this is possible only if $\lambda_1<\lambda< 2\lambda_0,$ where $\lambda_1 >\lambda_0.$
The top point of every geodesic with such parameter $\lambda,$ going through
$(0,1)$ and  intersecting $B_3,$ is $(|x(\lambda)|,\lambda)$ and
$|x(\lambda)|>|x(\lambda_1)|,$ while the most right point of $S_K$
is $(|x(\lambda_0)|,\lambda_0)$. Thus the shift of this geodesic to
the left of size less than
$$\nu=|x(\lambda_1)|-|x(\lambda_0)|=\frac{1}{2}\left(\lambda_1-\lambda_0+\frac{1}{\lambda_0}-\frac{1}{\lambda_1}\right)$$
will have the top point to the right of $S_K.$

Let $\eta=\max\{x|(x,1)\in B((0,1),K)\}$ and $\xi=\min(\nu,\eta).$
Using previous considerations, one can check that the set $P=D\cap
U((0,1),K),$ where $D$ is the set bounded above by curves
$$\lambda_0C_2+(-(x(\lambda_0)+\xi),0),\quad \lambda_0C_2+(x(\lambda_0)+\xi,0)$$
and below by curves
$$\lambda_0C_2+(x(\lambda_0)-\xi,0),\quad \lambda_0C_2+(\xi-x(\lambda_0),0),$$
(cf. formula (\ref{xlambda})) has the following properties:

1) $(0,1)\in \text{int}(P);$

2) For every point $Y\in \text{int}(P),$
every geodesic with parameter $\lambda\geq \lambda_0$
going through $Y$ intersects
the set $\{(x,1)|x\in (-\xi,\xi)\};$
and

3) For every point $Y\in \text{int}(P)$ every geodesic, going
through $Y,$  intersects the shere $S_K$ exactly at two (mutually
antipodal) points.

Since numbers $\xi$ and $\lambda_0$ continuously depend on $K,$ and
the  Stadium space is (metrically) homogeneous, we get the following
theorem.

\begin{thm}
\label{ust} The Stadium space $X$ is uniformly locally
$G$-homogeneous on $X.$  As a corollary, it is locally
$G$-homogeneous and uniformly locally $G$-homogeneous on an orbal
subset. On the other hand, $X$ has no convex ball of positive
radius.
\end{thm}

\section{Epilogue}

We begin by the following remark: All
statements of Sections 2--7, except for the
local uniqueness of joins and Proposition~3.3, 
can be generalized to spaces with 
distinguished geodesics (in the sense of [20]). 

In conclusion, let us note that the 
Busemann conjecture remains an important problem in the
characterization of manifolds. Proposition \ref{cone} and Corolary
\ref{homo} imply that it is equivalent to the statement that
sufficiently small metric spheres in a finite-dimensional Busemann
$G$-space are codimension one manifold  factors. We conclude the
paper by some questions.

\begin{que}
Is every Busemann $G$-space $X$ necessarily locally $G$-homogeneous
or uniformly locally $G$-homogeneous on an orbal subset?
\end{que}

\begin{que}
Is every sufficiently small sphere in $n$-dimensional Busemann
$G$-space  homotopy equivalent to the  $(n-1)$-sphere?
\end{que}

\begin{que}
Are there finite-dimensional locally $G$-homogeneous Busemann
$G$-spaces  with nonmanifold arbitrary small metric spheres?
\end{que}

A positive answer to the last question would provide an example of a
compact topologically homogeneous finite-dimensional nonmanifold ANR
which is a homology sphere having the property that the complement
of every one of its points is contractible.

\section*{Acknowledgements}

The authors thank I.A. Zubareva for useful discussions.
This research was supported by the State Maintenance Program
for the Leading Scientific Schools of the Russian Federation
(grant NSH-6613.2010.1), RFBR (grant 08-01-00067-a),
RFBR-BRFBR (grant 10-01-90000-Bel-a), the project
"Quasiconformal Analysis and geometric aspects of operator
theory",
the Brigham Young University
Special Topology Year 2008-2009 Fund and
the Slovenian Research Agency grants
BI-US/08-10/003,
P1-0292-0101, and
J1-2057-0101.
We acknowledge the referee for several comments and suggestions.

\end{document}